\documentclass[draft]{article}

\usepackage[T1]{fontenc}
\usepackage{lmodern}

\usepackage{mathtools,amsfonts,etoolbox}
\usepackage[numbers,sort&compress]{natbib}
\usepackage{amsthm}

\mathtoolsset{
  showonlyrefs,
  mathic, 
  centercolon=true
}
\DeclarePairedDelimiterX\abs[1]\lvert\rvert{\ifblank{#1}{\,\cdot\,}{#1}}
\DeclarePairedDelimiterX\norm[1]\lVert\rVert{\ifblank{#1}{\,\cdot\,}{#1}}
\DeclarePairedDelimiterXPP\onenorm[1]{}\lVert\rVert{_1}{\ifblank{#1}{\,\cdot\,}{#1}}
\DeclarePairedDelimiterXPP\cbnorm[1]{}\lVert\rVert{_{\mathrm{cb}}}{\ifblank{#1}{\,\cdot\,}{#1}}

\DeclarePairedDelimiterX\eval[2]\langle\rangle{%
  \ifblank{#1}{\,\cdot\,}{#1},\ifblank{#2}{\,\cdot\,}{#2}%
}

\newcommand\bded[2]{\mathcal{B}%
\ifstrequal{#1}{#2}{(#1)}{(#1, #2)}}

\newcommand\Z{\mathbb{Z}}
\newcommand\N{\mathbb{N}}

\newcommand\C{\mathbb{C}}

\newcommand\T{\mathbb{T}}
\newcommand\Dbar{\overline{\mathbb{D}}}

\newcommand\ri{\mathrm{i}}
\newcommand\re{\mathrm{e}}

\newcommand{\dm}{\mathop{}\!\mathrm{d}}

\newcommand\adj{^*}
\newcommand\dadj{^{**}}

\newcommand\dual{^*}
\newcommand\ddual{^{**}}
\newcommand\annih{^{\perp}}

\newcommand\inv{^{-1}}
\newcommand\id{\mathrm{id}}

\newcommand\charspace[1]{\Phi_{#1}}
\newcommand\pe[1]{\varepsilon_{#1}}

\newcommand\shilov[1]{\Gamma(#1)}
\newcommand\choquet[1]{\Gamma_0(#1)}

\DeclareMathOperator\supp{supp}
\DeclareMathOperator\tint{int}

\DeclareMathOperator\cconv{\overline{conv}}

\newcommand\pushforward[2]{{#1}_{\ast}(#2)}

\newtheorem{theorem}{Theorem}[section]

\newtheorem{proposition}[theorem]{Proposition}

\newtheorem{corollary}[theorem]{Corollary}
\newtheorem{lemma}[theorem]{Lemma}
\newtheorem{question}[theorem]{Question}

\theoremstyle{definition}
\newtheorem{definition}[theorem]{Definition}

\newtheorem*{acknowledgements}{Acknowledgements}

\theoremstyle{remark}
\newtheorem{remark}[theorem]{Remark}
\newtheorem{example}[theorem]{Example}

\begin{document}

\title{Extensions of uniform algebras}
\author{Sam Morley}


\maketitle
\begin{abstract}
The theory of algebraic extensions of Banach algebras is well established, and there are many constructions which yield interesting extensions. In particular, Cole's method for extending uniform algebras by adding square roots of functions to a given uniform algebra has been used to solve many problems within uniform algebra theory. However, there are numerous other examples in the theory of uniform algebras that can be realised as extensions of a uniform algebra, and these more general extensions have received little attention in the literature. In this paper, we investigate more general classes of uniform algebra extensions. We introduce a new class of extensions of uniform algebras, and show that several important properties of the original uniform algebra are preserved in these extensions. We also show that several well-known examples from the theory of uniform algebras belong to these more general classes of uniform algebra extensions.
\end{abstract}

Many important examples of uniform algebras can be realised as an extension of a given uniform algebra, such as the counterexamples to the peak point conjecture given by Cole \cite{cole1968} and Basener \cite{basener1971}. In the former, it is explicit that an extension of uniform algebra is constructed, whilst in the latter this fact is less obvious. In fact, it can be shown that both of these examples belong to an interesting class of extensions of uniform algebras, which we introduce in this paper. Extensions that belong to this class have numerous desirable properties and, perhaps most importantly, have the property that the extension is a non-trivial uniform algebra if the original uniform algebra was non-trivial (see Section~\ref{sec:preliminaries} for definitions). In this paper, we develop a general theory of uniform algebra extensions that will be useful in future constructions.

Much of the literature on extensions of uniform algebras---and, more generally, of commutative Banach algebras---has focused on algebraic extensions. (The study of algebraic extensions of commutative Banach algebras originates in work of Arens and Hoffman \cite{Arens1956}, see also \cite{lindberg1964,lindberg1973, dawson2003, karahanjan1978,narmaniya1982}.) These algebraic extensions of uniform algebras, such as Cole extensions (see Section~\ref{sec:Cole extensions}), and their properties are generally well-understood, while more general classes of uniform algebra extensions have received relatively little attention. Many results proven for for algebraic extensions of uniform algebras can be generalised to a larger class of extensions.

We are especially interested in the properties of a uniform algebra that are inherited by any uniform algebra extension, such as being non-trivial, natural, regular, or normal. We are also interested in the relationship between peak sets (in the weak sense) and the Shilov boundary of a uniform algebra extension to those of the original uniform algebra. Cole's construction yields a uniform algebra extension that preserves many of the properties of the original uniform algebra. This is largely due to the existence of a norm one linear map that maps the Cole extension onto the original algebra. Several examples in the theory of uniform algebras use Cole's construction, for example in \cite{feinstein2004,feinstein2001,feinstein1992}.

Many recent examples in uniform algebra theory can be realised---or are explicitly constructed---as uniform algebra extensions. For example, in \cite{dales2018}, Dales, Feinstein, and Pham construct a uniform algebra that contains a maximal idea in which all null sequences factor but does not contain a bounded approximate identity. In \cite{feinstein2015}, Feinstein and Izzo construct give a construction that yields essential uniform algebras.

In Section~\ref{sec:generalised Cole extensions}, we introduce a class of uniform algebra extensions, motivated by Cole's construction, which we call generalised Cole extensions. These extensions are distinguished by the existence of a norm one linear map onto the original uniform algebra. We show that these generalised Cole extensions preserve several properties of the original uniform algebra; in particular, they preserve non-triviality. The class of generalised Cole extensions is larger than the class of Cole extensions in general, although we do provide some special conditions under which a generalised Cole extension is necessarily a Cole extension (Theorem~\ref{thm when gce is ce}). These conditions are closely related to the solution to the ``bicontractive/symmetric projection problem'' for uniform algebras discussed in \cite{blecher2015}. We also show that Basener's counterexample to the peak point conjecture can also be realised as a generalised Cole extension of a certain uniform algebra.

The outline of the paper is as follows. We first fix our notation, terminology, and provide some elementary results on averaging operators. We the investigate general uniform algebra extensions (Section~\ref{sec:uniform algebra extensions}), followed by a short description of (simple) Cole extensions (Section~\ref{sec:Cole extensions}). Finally, we introduce our new class of uniform algebra extensions and discuss their properties (Section~\ref{sec:generalised Cole extensions}). We discuss how, in some cases, the action of a compact group yields the norm one linear map onto the original algebra. We show that, under certain conditions, this is always the case. Finally, we show that Basener's example belongs to this new class of extension, and that the norm one linear map arises due to the action of a compact group.

\section{Preliminaries}
\label{sec:preliminaries}
Throughout this paper, we say {\em compact space} to mean a non-empty, compact, Hausdorff topological space, we say {\em measure} to mean a regular, complex, Borel measure, and we denote the set of positive integers by $\N$.

Let $\mathfrak X$ be a Banach space, and let $F$ be a closed linear subspace of $\mathfrak X$. We write $\mathfrak X\dual$ for the {\em topological dual} of $\mathfrak X$, and we write $F\annih$ for the {\em annihilator} of $F$ in $\mathfrak X\dual$; that is $F\annih$ is the space of all linear functionals $\lambda\in \mathfrak X\dual$ such that $\lambda(F)\subseteq\{0\}$.

Let $A$ be complex, commutative, unital Banach algebra. We denote the character space of $A$ by $\charspace A$. We give $\charspace A$ the relative weak-$*$ topology so that $\charspace A$ is a compact space. We write $\widehat a$ for the Gelfand transform of $a\in A$, which is a continuous function from $\charspace A$ into $\C$; we write $\widehat A$ for the set $\{\widehat a:a\in A\}$, which is an algebra of continuous, complex-valued functions on $\charspace A$.

Let $X$ be a compact space. We denote the algebra (with pointwise operations) of all continuous, complex-valued functions on $X$ by $C(X)$. For each non-empty $E\subseteq X$, we write
\[
\abs f_E:=\sup_{x\in E}{\abs{f(x)}}\qquad(f\in C(X)).
\]
Let $S$ be a subset of $C(X)$. We say that $S$ {\em separates the points of $X$} if, for each $x,y\in X$ with $x\neq y$, there exists $f\in S$ such that $f(x)\neq f(y)$. It is standard that $\abs{}_X$ a norm on $C(X)$, and that $C(X)$ is a commutative Banach algebra with this norm. We call $\abs{}_X$ the {\em uniform norm}.
A {\em uniform algebra on $X$} is a closed subalgebra $A$ of $C(X)$ such that $A$ separates the points of $X$ and $A$ contains all constant functions on $X$. Let $A$ be a uniform algebra on $X$. We say that $A$ is {\em trivial} if $A=C(X)$, otherwise $A$ is {\em non-trivial}. For each $x\in X$, let $\pe x$ denote the {\em point evaluation character at $x$}, given by $\pe x(f) = f(x)$ ($f\in A$). The map $x\mapsto \pe x$ ($x\in X$) identifies $X$ with a closed subset of $\charspace A$. We say that a uniform algebra $A$ is {\em natural on $X$} if all characters on $A$ correspond to point evaluations at points of $X$; that is, if $X=\charspace A$.

Let $E$ be a non-empty closed subset of $X$. We write $A|E$ for the algebra $\{f|E: f\in A\}$. The uniform closure in $C(E)$ of $A|E$ is a uniform algebra on $E$, which we denote by $A_E$.

We refer the reader to \cite{gamelin1984,stout1971,browder1969} for additional information on uniform algebras, and to \cite{dales2000} for additional information on Banach algebras.

We also require the following definitions.

\begin{definition}
Let $A$ be a uniform algebra on a compact space $X$. We say that $A$ is {\em normal on $X$} if for each compact set $K\subseteq X$ and each closed set $E\subseteq X$ with $E\cap K=\emptyset$, there exists $f\in A$ with $f(K)\subseteq\{1\}$ and $f(E)\subseteq\{0\}$; we say that $A$ is regular on $X$ if for each closed set $E\subseteq X$ and each point $x\in X\setminus E$, there exists $f\in A$ such that $f(x) = 1$ and $f(E)\subseteq \{0\}$.
\end{definition}

Let $A$ be a uniform algebra on a compact space $X$. Then $A$ is normal if and only if $A$ is natural and regular on $X$. (See \cite[Theorem~27.2]{stout1971}.)

\begin{definition}
Let $A$ be a uniform algebra on a compact space $X$, and let $E\subseteq X$. We say that $E$ is a {\em boundary (for $A$)} if, for all $f\in A$, there exists $x\in E$ such that $\abs{f(x)}=\abs{f}_X$.
We say that $E$ is a {\em closed boundary (for $A$)} if $E$ is a closed set which is a boundary.
The intersection of all closed boundaries for $A$ is called the {\em Shilov boundary (for $A$)} and is denoted by $\shilov A$.
We say that $E$ is a {\em peak set (for $A$)} if there exists $f\in A$ such that $f(x)=1$ $(x\in E)$ and $\abs{f(x)}<1$ $(x\in X\setminus E)$; in this case, we say that the function $f$ {\em peaks on $E$}.
We say that $E$ is a {\em peak set in the weak sense (for $A$)} if there is a family $\mathcal F$ of peak sets for $A$ such that $E=\bigcap_{F \in \mathcal F}F$.
Let $x\in X$. We say that $x$ is a {\em peak point (for $A$)} if $\{x\}$ is a peak set, and we say that $x$ is a {\em strong boundary point (for $A$)} if $\{x\}$ is a peak set in the weak sense; the set of all strong boundary points for $A$ is called the {\em Choquet boundary (for $A$)} and is denoted by $\choquet A$.
\end{definition}

It is standard that $\choquet A$ is a boundary for $A$, that $\shilov A$ is a closed boundary for $A$, and that $\shilov A$ is the closure of $\choquet A$. 

Here, and throughout this paper, we shall identify the dual space of $C(X)$, where $X$ is a compact space, with the set of all measures on $X$. We need the following standard terminology.

\begin{definition}
Let $A$ be a uniform algebra on a compact space $X$, and let $\varphi$ be a character on $A$. A {\em representing measure for $\varphi$} (with respect to $A$) is a probability measure $\mu$ on $X$ such that
\[
\varphi(f) = \int_X f\dm\mu\qquad(f\in A).
\]
A {\em Jensen measure for $\varphi$} (with respect to $A$) is a probability measure $\mu$ on $X$ such that
\[
\log{\abs{\varphi(f)}}\leq \int_X \log{\abs{f}}\dm\mu\qquad(f\in A).
\]
An {\em annihilating measure} for $A$ is a measure $\mu$ on $X$ such that
\[
\int_X f\dm\mu = 0
\]
for all $f\in A$; the space of annihilating measures for $A$ is identified with $A\annih$.
\end{definition}

Let $A$ be a uniform algebra on a compact space $X$, let $E\subseteq X$, and let $\varphi$ be a character on $A$. By Theorems~2.4.7 and 2.4.9 of \cite{browder1969}, $E$ is a peak set in the weak sense for $A$ if and only if, for every annihilating measure $\mu$ for $A$, the measure $\mu_E$ (given by $\mu_E(S) = \mu(E\cap S)$ for Borel sets $S\subseteq X$) is also an annihilating measure for $A$.   Every Jensen measure for $\varphi$ is a representing measure for $\varphi$, and $\varphi$ always admits a Jensen measure. We say that a Jensen measure $\mu$ for $\varphi$  is {\em trivial} if there exists $x\in X$ such that $\mu$ is the point-mass measure at $x$. In this case, $\varphi = \pe x$ for all $f\in A$.

\smallskip

Let $X$ and $Y$ be topological spaces, let $f:X\to Y$ be a continuous function, and let $F$ be a family of continuous functions from $Y$ into a topological space $Z$. We write $f\adj$ for the function $g\mapsto g\circ f$, which maps $F$ into the set of continuous functions from $X$ into $Z$. If $X$ and $Y$ are Banach spaces, $f$ is linear, $Z=\C$, and $F$ is a collection of bounded linear functionals on $Y$, then $f\adj$ is the usual adjoint operator. In the case where $X$ and $Y$ are compact spaces, $f\adj:C(Y)\to C(X)$ is a continuous, unital homomorphism with $\norm{f\adj}=1$. Moreover, if $f$ is surjective then $f\adj$ is isometric.

Let $X$ be a compact space and let $\mu$ be a measure on $X$. We write $\supp\mu$ for the {\em support} of $\mu$; that is, $\supp\mu$ is the complement of the largest open subset $U$ of $X$ such that $\abs{\mu}(U)=0$. Let $Y$ be a second compact space. Suppose that ${\Pi:X\to Y}$ is a continuous surjection. Then the {\em pushforward measure} $\pushforward\Pi\mu$, defined by $\pushforward\Pi\mu(E)=\mu(\Pi\inv(E))$ for each Borel set $E\subseteq Y$, is a regular, complex Borel measure on $Y$. It is useful to note that ${\Pi\dadj:C(X)\dual\to C(Y)\dual}$ (where $\Pi\dadj = (\Pi\adj)\adj$) is the map which sends a measure $\mu$ on $Y$ to the pushforward measure $\pushforward\Pi\mu$ on $X$. (This fact is noted in \cite[Example~VI.1.7]{conway1990a}.) In particular, if $f:Y\to\C$ is a Borel measurable function, then $f\circ\Pi$ is Borel measurable and $f$ is $\pushforward\Pi\mu$-integrable if and only if $f\circ\Pi$ is $\mu$-integrable, and if this occurs then
\begin{equation}\label{pushforward integral equality}
\int_X f \dm\pushforward\Pi\mu = \int_Y f\circ \Pi\dm\mu.
\end{equation}
(See, for example, \cite[Theorem~2.4.18]{federer1969}.)

Let $E\subseteq\C$. We write $\cconv{E}$ for the {\em closed convex hull of $E$}; that is, $\cconv{E}$ is the smallest closed, convex subset of $\C$ which contains $E$.

\smallskip

We conclude this section with an elementary lemma; we include a short proof for the convenience of the reader.

\begin{lemma}\label{fibre neighbourhood lemma}
Let $X$ and $Y$ be compact spaces and let $K$ be a closed subset of $X$. Suppose that $\Pi:Y\to X$ is a continuous surjection and $V$ is an open neighbourhood of $\Pi\inv(K)$ in $Y$. Then there exists an open neighbourhood $U$ of $K$ in $X$ such that $\Pi\inv(U)\subseteq V$.
\end{lemma}
\begin{proof}
Let $(N_\alpha)$ be the family of all closed neighbourhoods of $K$ in $X$. For each $\alpha$, let $V_\alpha :=Y\setminus \Pi\inv(N_\alpha)$. We have that $\bigcap_{\alpha} N_\alpha = K$, and so $Y\setminus U\subseteq \bigcup_{\alpha}V_\alpha$. Since $Y\setminus U$ is compact, there exists $\alpha_1,\dots,\alpha_n$ such that $Y\setminus U\subseteq\bigcup_{j=1}^n V_{\alpha_j}$. Set $N := \bigcap_{j=1}^n N_{\alpha_j}$ and let $U$ be an open set in $X$ with $K\subseteq U \subseteq N$. Then $U$ has the desired properties.
\end{proof}

\subsection{Averaging operators}
We first examine the properties of certain linear operators between algebras of continuous functions. Before we give any formal definitions, we first fix some notation. Throughout this section let $X$ and $Y$ be compact spaces, and let ${T:C(Y)\to C(X)}$ be a bounded linear operator. If $\lambda\in {C(X)}\dual$ then there exists a unique measure $\mu_\lambda$ on $Y$ such that
\begin{equation}\label{averaging measure}
\lambda(T(f))=\int_Y f(y)\dm\mu_\lambda(y)
\end{equation}
for all $f\in C(Y)$. If $\lambda$ is a character on $C(X)$, then $\lambda$ is a point evaluation at some point $x\in X$, and we instead write $\mu_x$ in place of $\mu_\lambda$ for the measure associated with this character.

Throughout the remainder of this paper, the measures $\mu_\lambda$ will always denote the measure associated (via \eqref{averaging measure}) with the linear operator labelled $T$, whenever this operator is defined.

We say that $T$ is {\em positive} if, for all $f\in C(Y)$ such that $f(Y)\subseteq[0,\infty)$, we have $T(f)(X)\subseteq[0,\infty)$. In this case, the measure $\mu_x$, as above, is a positive measure for all $x\in X$. It is standard that if $T$ is unital and $\norm{T}=1$ then $T$ is positive (see, for example, \cite[p.~17]{blecher2004b}).

We require the following definition from \cite{kelley1958}.

\begin{definition}
Suppose that $Y=X$. We say that $T$ is an {\em averaging operator} if, for each $x\in X$, we have
\[
\supp\mu_{x}\subseteq \{y\in X:T(f)(y)=T(f)(x)\quad (f\in C(Y))\}.
\]
\end{definition}

Kelley \cite[Theorem~2.2]{kelley1958} showed that, in the case where $Y=X$, a bounded linear operator $T$ is averaging if and only if ${T(fT(g))=T(f)T(g)}$ for all $f,g\in C(X)$. It is this property that is most useful to us. Note that Kelley's proof was for the case where $C(X)$ is the algebra of continuous, real-valued functions on $X$. However, it is a simple task to see that the proof works for complex-valued functions too (indeed Kelley notes this in the paper).

We now expand on this result. Several of the statements contained within the following lemma are well-known, but we include details for the convenience of the reader.

\begin{lemma}\label{averaging equivalences}
Suppose that $\Pi:Y\to X$ is a continuous surjection with $T\circ\Pi\adj=\id_{C(X)}$. The following statements are equivalent:
\begin{enumerate}
  \item\label{averaging equivalences a} the projection $\Pi\adj\circ T:C(Y)\to C(Y)$ is averaging;
  \item\label{averaging equivalences b} for each $f\in C(Y)$ and $g\in C(X)$ we have $T(f\Pi\adj(g))=T(f)g$;
  \item\label{averaging equivalences c} for each $x\in X$ we have $\supp\mu_x\subseteq \Pi\inv(\{x\})$.
\end{enumerate}
Suppose further that $T$ is unital and that $\norm{T}=1$. Then $T$ is positive and satisfies each of the above conditions. In this case, $T$ also satisfies the following
\begin{enumerate}\setcounter{enumi}{3}
    \item\label{averaging equivalences d} for each $f\in C(Y)$ and $x\in X$ we have $T(f)(x)\in\cconv{f(\Pi\inv(\{x\}))}$.
\end{enumerate}
\end{lemma}
\begin{proof} Throughout, set $P:=\Pi\adj\circ T$. Then $P$ is a continuous projection on $C(Y)$ whose range is precisely the set $\Pi\adj(C(X))$.
\item \ref{averaging equivalences a} $\implies$ \ref{averaging equivalences b}. Suppose that $P$ is averaging. Then, by \cite[Theorem~2.2]{kelley1958}, $P$ satisfies ${P(fP(h))=P(f)P(h)}$ for all ${f,h\in C(Y)}$. Fix ${f\in C(Y)}$ and ${g\in C(X)}$; set ${h:=\Pi\adj(g)}$. The map $\Pi\adj$ is an isometric homomorphism from $C(X)$ into $C(Y)$, so we have
    \[
    P(fP(h))=P(f)P(h)=\Pi\adj(T(f))\Pi\adj(T(h))= \Pi\adj(T(f)T(h)),
    \]
    so that $\Pi\adj(T(fT(h)))=\Pi\adj(T(f)T(h))$. Thus $T(f\Pi\adj(T(h)))=T(f)T(h)$. But then $T(h)=T(\Pi\adj(g))=g$ so that $T(fh)=T(f)g$. This proves \ref{averaging equivalences b}.
\item \ref{averaging equivalences b} $\implies$ \ref{averaging equivalences a}. Suppose that $T(f\Pi\adj(g))=T(f)g$ for all $f\in C(Y)$ and $g\in C(X)$. We have
    \[
    P(fP(\Pi\adj(g)))=P(f\Pi\adj(g))=\Pi\adj(T(f\Pi\adj(g)))= \Pi\adj(T(f)g)=P(f)\Pi\adj(g),
    \]
    for all $f\in C(Y)$ and $g\in C(X)$. However, the range of $P$ is precisely the subalgebra $\Pi\adj(C(X))$ of $C(Y)$, and so for each $h\in C(Y)$ there exists a unique $g\in C(X)$ such that $P(h)=\Pi\adj(g)$. Thus, by the above, ${P(fP(h))=P(f)P(h)}$ for all ${f, h\in C(Y)}$. By \cite[Theorem~2.2]{kelley1958}, $P$ is averaging. This proves \ref{averaging equivalences a}.
\item \ref{averaging equivalences a} $\iff$ \ref{averaging equivalences c}. For each $y\in Y$, let $\nu_y$ be the unique, regular Borel measure on $Y$ such that, for all $f\in C(Y)$, we have
    \[
    P(f)(y)=\int_Y f\dm\nu_y.
    \]
    Fix $x\in X$ and let $y\in Y$ such that $\Pi(y)=x$. Let
    \[
    E:=\{z\in Y:P(f)(z)=P(f)(y)\quad (f\in C(Y))\}.
    \]
    Then we have
    \[
    \int_Y f\dm\nu_y=(\Pi\adj\circ T)(f)(y)=T(f)(\Pi(y))=\int_Y f\dm\mu_{x}
    \]
    for all $f\in C(Y)$, so that $\nu_y=\mu_{x}$.
    If $z\in \Pi\inv(\{x\})$ then we have $\Pi(z)=x$ and so
    \[
    P(f)(z)=T(f)(\Pi(z))=T(f)(x)=P(f)(y).
    \]
    Thus $\Pi\inv(\{x\})\subseteq E$. Conversely, if $z\in E$ then we have
    \[
    T(f)(\Pi(z))=P(f)(z)=P(f)(y)=T(f)(x),
    \]
    for all $f\in C(Y)$. Thus we must have $\Pi(z)=x$; otherwise, choose a continuous function $g\in C(X)$ such that $g(\Pi(z))=0$ and $g(x)=1$ and look at $f=\Pi\adj(g)$. Thus we have $E\subseteq \Pi\inv(\{x\})$ and equality holds.

    If $P$ is averaging then we have
    \[
    \supp\mu_x = \supp\nu_y \subseteq E =\Pi\inv(\{\Pi(y)\})=\Pi\inv(\{x\}).
    \]
    Conversely, if $\supp\mu_x \subseteq \Pi\inv(\{x\})=E$ then we have
    \[
    \supp\nu_y \subseteq E
    \]
    so that $P$ is averaging. This proves that \ref{averaging equivalences a} holds if and only if \ref{averaging equivalences c} holds.

\smallskip
For the remainder of the proof assume, in addition, that $T$ is unital and $\norm{T}=1$. It follows that $T$ is positive. Then, for each $x\in X$ the measure $\mu_x$ is positive with $\mu_x(Y)=1$.

    We must show that $P$ is averaging under these additional assumptions. Since $T$ is unital and $\norm{T}=1$, it follows that $P$ is a unital projection on $C(Y)$ such that $\norm{P}=1$. Then, by a well-known theorem of Tomiyama \cite{tomiyama1957} (see also \cite[Corollary~4.2.9]{blecher2004b}), $P$ satisfies $P(P(f)gP(h))=P(f)P(g)P(h)$ for all $f,g,h\in C(Y)$. Thus $P$ is averaging by \cite[Theorem~2.2]{kelley1958}.

It remains to show that the above conditions imply that $T$ satisfies \ref{averaging equivalences d}.

\item \ref{averaging equivalences c} $\implies$ \ref{averaging equivalences d}. Suppose that $\supp\mu_x\subseteq\Pi\inv(\{x\})$ for all $x\in X$ and $f\in C(Y)$. Then, since $\mu_x$ is a probability measure, we have
    \[
    T(f)(x)=\int_Y f(y)\dm\mu_x(y)\in\cconv{f(\Pi\inv(\{x\}))}
    \]
    for all $x\in X$ and $f\in C(Y)$.

This completes the proof.
\end{proof}

The extra assumptions that $T$ is unital and satisfies $\norm{T}=1$ are not a significant obstruction to the topics discussed in this paper.

The map $x\mapsto \mu_x:X\to C(Y)\dual$ is weak-$*$ continuous as noted in \cite{lloyd1963,kelley1958}, and any weak-$*$ continuous function $\psi$ from $X$ into $C(Y)\dual$ yields an bounded linear map from $C(Y)$ into $C(X)$.

\begin{remark}
Property \ref{averaging equivalences b} in the above lemma is reminiscent of the defining property of ``conditional expectations'', which are important in $C^\ast$-algebra and von Neumann algebra theory.
\end{remark}

\section{Uniform algebra extensions}
\label{sec:uniform algebra extensions}
We begin by examining the general class of uniform algebra extensions of a given uniform algebra. Many results proved in special cases can be generalised to the general setting. Throughout, we shall use the following definition of a uniform algebra extension.

\begin{definition}
Let $X, Y$ be compact spaces, let $A$ be a uniform algebra on $X$, and let $B$ be a uniform algebra on $Y$. Suppose that there exists a continuous surjection $\Pi:Y\to X$ such that $\Pi\adj(A)\subseteq B$. Then we say that $B$ is a {\em uniform algebra extension} of $A$, and $\Pi$ is the map {\em associated} with this extension.
\end{definition}

Note that, with $X$, $Y$, $A$, and $B$ as in the above, there may be many continuous surjections $\Pi:Y\to X$ for which $B$ is a uniform algebra extension of $A$ (in the sense defined above).

Before we continue, we make some remarks about our choice of terminology. Let $A$, $B$, $X$, and $Y$ be as in the definition. If $\theta:A\to B$ is an isometric homomorphism then it is not hard to see that $\shilov A$ is always contained in the image of $\shilov B$ under the associated map $\tau_\theta:\charspace B\to\charspace A$. In particular, the image of $Y$ under $\tau_\theta$ is a compact subset of $\charspace A$ which contains $\shilov A$. Hence $A$ can be identified with a uniform algebra on $\tau_\theta(Y)$. By changing our $X$ to coincide with $\tau_\theta(Y)$, we can obtain a continuous surjection from $Y$ onto $X$.

There are many examples of uniform algebra extensions in the literature. For example, given any uniform algebra $A$ on a compact space $X$, the uniform algebra $C(X)$ is a uniform algebra extension of $A$, in the sense above, where the associated map is taken to be the identity map on $X$. In particular, if $A$ is a non-trivial uniform algebra and $B$ is any uniform algebra extension of $A$, then $B$ need not be non-trivial.

The following illustrative examples provide a basis for the present study.

\begin{example}\label{eg identify prefect set}
Let $K$ be a perfect subset of $Y$. Suppose that $X$ is obtained from $Y$ by contracting $K$ to a point, say $x_0\in X$, and that $\Pi$ is the quotient map. Let $A_0$ be a non-trivial uniform algebra on $K$. Let ${A=C(X)}$ and 
\[
{B=\{f\in C(Y):f|K\in A_0\}}.
\]
Clearly $B$ is a uniform algebra extension of $A$ and, since ${B_{x_0}=A_0\neq C(E_x)}$, we have $B\neq C(Y)$. Moreover, since every function in $\Pi\adj(C(X))$ is constant on $K$, it follows quickly that $B$ is natural on $Y$ if and only if $A_0$ is natural on $K$. 
\end{example}

Feinstein and Izzo \cite{feinstein2015} have recently used similar ideas to construct examples of essential uniform algebras. Note that, in the above, $A$ is natural on $X$, but $B$ need not be natural on $Y$.

\begin{example}
  Let $X$ and $Y$ be the closed unit disk in $\C$, let $\Pi$ be the identity map on $X$, let $A$ be the disk algebra, and let
  \[
    B = \{f\in C(Y):f|\T \in A|\T\},
  \]
  where $\T$ denotes the unit circle in $\C$. Then $B$ is not natural on $Y$; $\charspace B$ consists of two copies of the closed unit disk that are identified on $\T$. Note that, in this case, $\shilov B\not\subseteq \Pi\inv(\shilov A)$, and that $A$ admits non-trivial Jensen measures supported on $\T$, which are also Jensen measures for $B$.
\end{example}

Throughout the remainder of this section, let $X$ and $Y$ be compact spaces, let $A$ be a uniform algebra on $X$, let $B$ be a uniform algebra on $X$. We also assume that $B$ is a uniform algebra extension of $A$ with associated map $\Pi$. For each $x\in X$, we write $E_x$ for the fibre of $\Pi$ over $x$ (that is, $E_x:=\Pi\inv(\{x\})$), we write $B_x$ for $B_{E_x}$.

We have already seen that there are non-natural uniform algebra extensions of a natural uniform algebra. However, if $A$ is natural then $\widehat B$ is a uniform algebra extension of $A$ with associated map $\widehat\Pi:\charspace B\to X$ given by $\widehat\Pi(\varphi) = \varphi\circ\Pi\adj$ $(\varphi\in\charspace B)$. (Note that it is essential that $A$ is natural on $X$ to ensure that $\widehat\Pi$ is a surjection onto $X$.) 

On the other hand, if $B$ is natural on $Y$ then $B_x$ is natural on $E_x$ for each $x\in X$. To see this, for $x\in X$, note that the character space of $B_x$ is identified with the set of $y\in Y$ such that $\abs{f(y)}\leq \abs{f}_{E_x}$ for each $f\in B$. But then if $y\in Y\setminus E_x$, then there is a function $g\in\Pi\adj(A)$ such that $g(y) = 1$ and $g(E_x)\subseteq \{0\}$, so that $y$ does not belong to $\charspace{B_x}$. Thus $\charspace{B_x}\subseteq E_x$ and so equality holds.

Our first theorem provides stronger assumptions that ensure that $B$ is a natural uniform algebra. We first prove the following lemma, which is surely well-known. (A similar result is given in \cite[Lemma~8.1]{gamelin1978}. The result is also used in the proof of \cite[Lemma~9]{basener1971}.)

\begin{lemma}\label{Jensen pushforward}
Let $\varphi\in\charspace B$ and let $\mu$ be a Jensen measure for $\varphi$ with respect to $B$. Then the pushforward measure $\pushforward\Pi\mu$ is a Jensen measure for $\varphi\circ\Pi\adj$ with respect to $A$ and $\supp\mu\subseteq\Pi\inv(\supp\pushforward\Pi\mu)$.
\end{lemma}
\begin{proof}
Set $\psi:=\varphi\circ\Pi\adj$. It is easy to see that $\psi\in\charspace A$.
Fix $f\in A$, and let $g=\Pi\adj(f)$. Then, by \eqref{pushforward integral equality}, we have
\[
\log{\abs{\psi(f)}}=\log{\abs{\varphi(g)}}\leq \int_Y \log{\abs{g(y)}}\dm\mu(y)=\int_X \log{\abs{f(x)}}\dm\pushforward\Pi\mu,
\]
and so $\pushforward\Pi\mu$ is a Jensen measure for $\psi$. It is not hard to see that, in this case, we have ${\supp{\mu}\subseteq\Pi\inv(\supp{\pushforward\Pi\mu})}$. This completes the proof.
\end{proof}

We now prove our first theorem.

\begin{theorem}\label{Jensen measure triviality thm}
Suppose that $A$ does not admit any non-trivial Jensen measures and that, for each $x\in X$, $B_x$ does not admit any non-trivial Jensen measures. Then $B$ does not admit any non-trivial Jensen measures. In particular, $B$ is natural on $Y$.
\end{theorem}
\begin{proof}
Fix a character $\varphi$ on $B$, and let $\mu$ be a Jensen measure on $Y$ for $\varphi$ with respect to $B$. By Lemma~\ref{Jensen pushforward}, $\pushforward\Pi\mu$ is a Jensen measure on $X$ for $\varphi\circ\Pi\adj$, and we have $\supp\mu\subseteq\Pi\inv(\supp\pushforward\Pi\mu)$. Since $\pushforward\Pi\mu$ is trivial, there exists $x\in X$ such that $\pushforward\Pi\mu$ is the point mass measure at $x$. It follows that $\supp\mu\subseteq E_x$.

It is not hard to see that there exists $\psi\in\charspace{B_x}$ such that ${\psi(f|E_x)=\varphi(f)}$ for all $f\in B$, and that $\mu_{E_x}$ is a Jensen measure of $\psi$. Since $B_x$ does not admit any non-trivial Jensen measures, it follows that $\mu$ is trivial. This completes the proof.
\end{proof}

In many examples utilising uniform algebra extensions (explicity or implicitly), the base algebra $A$ is chosen to be regular (normal) on $X$, and are subject to the condition that $B_x = C(E_x)$ for each $x\in X$; the above theorem shows that $B$ is natural on $Y$ in this case.

Next we seek to understand the relationship between $\shilov A$ and $\shilov B$. Since $\Pi\adj$ is an isometry, it is easy to see that $\shilov A\subseteq \Pi(\shilov B)$. We start with the following proposition.

\begin{proposition}\label{peak set permanence}
  Let $E\subseteq X$ and let $F = \Pi\inv(E)$.
  \begin{enumerate}
    \item\label{peak set permanence a} If $E$ is a peak set for $A$, then $F$ is a peak set for $B$.
    \item\label{peak set permanence b} If $E$ is a peak set in the weak set for $A$, then $F$ is a peak set in the weak sense for $B$.
     \item\label{peak set permanence c} If $x\in \choquet A$ and $y\in \choquet{B_x}$ then $y\in \choquet B$.
  \end{enumerate}
\end{proposition}
\begin{proof}
  \ref{peak set permanence a} Suppose that $E$ is a peak set for $A$. Then it is easy to see that, if $f$ is a function which peaks on $E$, then $\Pi\adj(f)$ is a function which peaks on $F$ so that $F$ is a peak set for $B$.

  \ref{peak set permanence b} Suppose that $E$ is a peak set in the weak sense for $A$. Then there exists a family $\mathcal F$ of peak sets for $A$ such that $E=\bigcap_{C\in\mathcal F}C$. By part \ref{peak set permanence a}, we have ${\mathcal G:=\{\Pi\inv(C):C\in\mathcal F\}}$ is a family of peak sets for $B$, and ${\bigcap_{D\in\mathcal G}D=F}$. Thus $F$ is a peak set in the weak sense for $B$. This proves \ref{peak set permanence b}.

  \ref{peak set permanence c} This follows immediately from \ref{peak set permanence b} and \cite[Theorem~2.4.4]{browder1969}.
\end{proof}

It is certainly not true that, if $y$ is a strong boundary point for $B$ then $\Pi(y)$ is a strong boundary point for $A$; Cole's counterexample to the peak point conjecture does not have this property. We may restate \ref{peak set permanence c} as the following.

\begin{corollary}\label{Choquet boundary pullback restriction}
For each $x\in \choquet A$, we have $\choquet{B_x} = \choquet{B}\cap E_x$.
\end{corollary}

We now show that a similar relationship holds between $\shilov A$ and $\shilov B$. Our proof below is based on a similar proof due to Lindberg \cite[Theorem~3.1]{lindberg1964}.

\begin{theorem}\label{shilov boundary extensions}
Suppose that $\Pi$ is an open map. Then for each $x\in\shilov A$, we have $\shilov{B_x} \subseteq \shilov{B}\cap E_x$.
\end{theorem}
\begin{proof}
Fix $x\in\shilov A$, let $y_0\in \shilov{B_x}$, and let $U$ be an open neighbourhood of $y_0$ in $Y$. We first show that, if $E_x\subseteq U$, then there exists $f\in B$ such that $\abs{f}_Y = 1 > \abs{f}_{Y\setminus U}$. By Lemma~\ref{fibre neighbourhood lemma}, there is an open neighbourhood $V$ of $x$ in $X$ such that $E_x\subseteq \Pi\inv(V)\subseteq U$. Since $x\in\shilov A$, there exists  $g\in A$ such that $\abs{g}_X = 1$ and $\abs{g}_{X\setminus V}<1$. Set $f:=\Pi\adj(g)$. Then $f\in B$, $\abs{f}_Y = 1$ and $\abs{f}_{Y\setminus U}\leq \abs{f}_{Y\setminus\Pi\inv(V)}<1$. In particular, if $E_x = \{y_0\}$, then $y_0\in \shilov{B}$. Hence we may assume that $E_x$ contains at least two points, and that $E_x\setminus U$ is non-empty.

Let $\varepsilon\in (0,1/12)$ and let $U_0:=U\cap E_x$. Then $U_0$ is an open neighbourhood of $y_0$ in $E_x$. Since $y_0\in\shilov{B_x}$, there exists $f_0\in B_x$ such that $\abs{f_0}_{E_x} = 1$ and $\abs{f_0}_{E_x\setminus U_0}< 1$. By replacing $f_0$ with $f_0^n$ for sufficiently large $n\in\N$, we may assume that $\abs{f_0}_{E_x}=1$ and $\abs{f_0}_{E_x\setminus U_0}<\varepsilon$.

Now take $f_1\in B$ such that $\abs{f_1 - f_0}_{E_x} < \varepsilon$. Then we have $\abs{f_1}_{E_x\setminus U}<2\varepsilon$ and 
\[
1-\varepsilon < \abs{f_1}_{E_x} < 1 + \varepsilon.
\]
For each $y\in E_x$, let $U_y$ be an open neighbourhood of $y$ in $Y$ such that, for all $z\in U_y$, we have
\[
\abs{f_1(z) - f_1(y)} < \varepsilon.
\]
Fix $z_0\in E_x$ such that $\abs{f_1(z_0)} = \abs{f_1}_{E_x}$. Then, for each $y\in U_{z_0}$, we have
\[
\abs{f_1(y)} \geq \abs{f_1(z_0)} - \abs{f_1(y) - f_1(z_0)} > 1 - 2\varepsilon.
\]
Note that we must have $z_0\in U$ and, replacing $U_{z_0}$ with $U_{z_0}\cap U$, we may assume that $U_{z_0}\subseteq U$.

By compactness of $E_x\setminus U$, there exists $y_1,\dots, y_m\in E_x\setminus U$ such that
\[
E_x\setminus U \subseteq \bigcup_{j=1}^m U_{y_j}.
\]
Let $V := \Pi(U_{z_0})\cap \bigcap_{j=1}^m \Pi(U_{y_j})$. Since $\Pi$ is an open map, $V$ is an open neighbourhood of $x$ in $X$. By Lemma~\ref{fibre neighbourhood lemma}, there exists an open neighbourhood $W$ of $x$ in $X$ such that $W\subseteq V$ and 
\[
E_x\subseteq \Pi\inv(W) \subseteq U\cup \bigcup_{j=1}^m U_{y_j}.
\]
Since $x\in\shilov{A}$, there exists $g\in A$ such that $\abs{g}_X = 1 > \abs{g}_{X\setminus W}$. As above, replacing $g$ with $g^n$ for some $n\in\N$ if necessary, we may assume that $\abs{g}_X=1$ and $\abs{g}_{X\setminus W}<\varepsilon/\abs{f_1}_Y$. Let $W_0$ be the subset of $X$ consisting of all $x'\in X$ such that $\abs{g(x')} > 1-\varepsilon$ and set $f:=\Pi\adj(g)f_1\in B$. Note that $W_0$ is an open set, $W_0\subseteq W$, and $\Pi\inv(W_0)\cap U_{z_0}\neq\emptyset$.

Let $y\in Y$. If $y\in Y\setminus \Pi\inv(W)$ then we have
\[
\abs{f(y)} \leq \abs{g(\Pi(y))}\abs{f_1}_Y < \varepsilon.
\]
If $y\in U_{y_k}$ for some $k\in\{1,\dots,m\}$, then we have
\[
\abs{f(y)} \leq \abs{g}_X(\abs{f_1(y) - f_1(y_j)} + \abs{f_1(y_j)}) < 2\varepsilon.
\]
In particular, if $y\in Y\setminus U$ then we have $\abs{f(y)}<2\varepsilon$.

If $y\in U_{z_0}\cap \Pi\inv(W_0)$ then we have
\[
\abs{f(y)} = \abs{g(\Pi(y))}\abs{f_1(y)} > (1-\varepsilon)(1-2\varepsilon) > 1-3\varepsilon.
\]
Since $U_{z_0}\cap \Pi\inv(W_0)\neq\emptyset$, we have $\abs{f}_Y>1-3\varepsilon > 3/4$ and, by the above, we also have $\abs{f}_{Y\setminus U} < 2\varepsilon<1/6$. It follows that $y_0\in \shilov{B}$.
\end{proof}

\begin{remark}
More is true in the above. If $x\in \shilov{A}$ and $y\in\shilov{B}\cap\tint_{Y}(E_x)$ then $y\in\shilov{B_x}$. A corollary of the local maximum modulus theorem, \cite[Corollary~6.2]{rossi1960}, suggests that $\shilov{B}\cap E_x\subseteq \shilov{B_x}$. However, this corollary is given without proof, and we have not been able to provide a proof for this inclusion.
\end{remark}

In the special case where $\shilov{B_x}=E_x$ for all $x\in\shilov x$, we have the following.

\begin{corollary}
Suppose that $\Pi$ is an open map. If we have $\shilov{B_x}=E_x$ for all $x\in \shilov{A}$ then $\shilov{B}\supseteq\Pi\inv(\shilov{A})$. In particular, this occurs if $B_x=C(E_x)$ for all $x\in X$.
\end{corollary}

\section{Motivating example: Cole extensions}
\label{sec:Cole extensions}
Our aim is to introduce a new class of uniform algebra extensions, which preserve certain properties of the original uniform algebra. Our motivating example is Cole's construction of a counterexample to the peak point conjecture, where he constructs extensions of a given uniform algebra by adjoining square roots to a given function. In this section, we describe a modified construction---based on Cole's method---for adjoining roots to an arbitrary monic polynomial over a given uniform algebra. The most general case can be reconstructed by considering systems of such extensions. We refer the reader to \cite{dawson2003,cole1968} and \cite[Section~19]{stout1971} for additional information and references.

We define the set
\[
X^q := \{(x,z)\in X\times \C: h_0(x)+h_1(x)z+\dots+h_{n-1}(x)z^{n-1} + z^n=0\}.
\]
It is not hard to see that $X^q$ is a compact space (see \cite{dawson2003}). Let $\Pi_q:X^q\to X$ and ${p_q:X\to\C}$ denote the restrictions of the coordinate projections; the map $\Pi_q$ is continuous, surjective, and open. Let $A^q$ denote the uniform closure in $C(X^q)$ of the algebra generated by the functions $p_q$ and $\Pi_q\adj(f)$ ($f\in A$). It is not hard to see that $A^q$ is a uniform algebra on $X^q$. We summarise some of the properties of $A^q$ in the following proposition.

\begin{proposition}\label{simple Cole extension props}%
\begin{enumerate}
	\item\label{simple Cole extension props a} There exists a bounded linear map ${T:C(X^q)\to C(X)}$ with $\norm{T}=1$ such that $T(A^q)=A$ and $T\circ\Pi_q\adj =\id_{C(X)}$.
	\item\label{simple Cole extension props b} We have $A^q|\Pi_q\inv(\{x\}) = C(\Pi_q\inv(\{x\}))$ for all $x\in X$.
	\item\label{simple Cole extension props c} If $A$ is natural on $X$ then $A^q$ is natural on $X^q$.
	\item\label{simple Cole extension props d} We have $\shilov{A^q}=\Pi_q\inv(\shilov{A})$.
\end{enumerate}
\end{proposition}

All of the above properties above are proved by Cole \cite{cole1968} in the case where ${q(t) = t^2-h}$ for some $h\in A$. Cole's proof of \ref{simple Cole extension props a} used an integral over a compact group; the general case is stated in \cite{dawson2003} (see also \cite{narmaniya1982}). Part \ref{simple Cole extension props b} follows immediately from the fact that $\Pi\inv(\{x\})$ contains at most $n$ points. The general case for parts \ref{simple Cole extension props c} and \ref{simple Cole extension props d} are proved in \cite{lindberg1973}. Many other properties of $A$ which are preserved in $A^q$ are mentioned in \cite{dawson2003}.

\section{Generalised Cole extensions}
\label{sec:generalised Cole extensions}
In this section we introduce our new class of extensions, which were originally described in correspondence between Cole and Feinstein. We shall see that many of the results from Section~\ref{sec:uniform algebra extensions} can be strengthened in this new setting and that any extension, in this class, of a non-trivial uniform algebra will be a non-trivial uniform algebra. Throughout this section, let $X$ be a compact space and let $A$ be a uniform algebra on $X$.

\begin{definition}
Let $B$ be a uniform algebra on a compact space $Y$. We say that $B$ is a {\em generalised Cole extension} of $A$ if there exists a continuous surjection $\Pi:Y\to X$ and a unital linear map $T:C(Y)\to C(X)$ with $\norm{T}=1$ such that $T\circ\Pi\adj=\id_{C(X)}$ and $T(B)=A$.
\end{definition}

The conditions $T(B)=A$ and $T\circ\Pi\adj=\id_{C(X)}$ in the above together imply that $\Pi\adj(A)\subseteq B$. The maps $\Pi$ and $T$ in the above will often be referred to as the maps {\em associated} with the extension.

We need not assume any additional conditions on the map $T$. By an application of Lemma~\ref{averaging equivalences} and elementary functional analysis, we see that $T$ has many desirable properties that are summarised in the following lemma, along with some additional facts that will be useful throughout.

\begin{lemma}\label{properties of map back GCE}
Let $B$ be a uniform algebra on a compact space $Y$. Suppose that $B$ is a generalised Cole extension of $A$ with associated maps $\Pi:Y\to X$ and $T:C(Y)\to C(X)$. Then $T$ satisfies each of the following conditions:
\begin{enumerate}
  \item\label{properties of map back GCE a} the projection $\Pi\adj\circ T:C(Y)\to C(Y)$ is averaging;
  \item\label{properties of map back GCE b} for each $f\in C(Y)$ and $g\in C(X)$ we have $T(f\Pi\adj(g))=T(f)g$;
  \item\label{properties of map back GCE c} for each $x\in X$ we have $\supp\mu_x\subseteq\Pi\inv(\{x\})$;
  \item\label{properties of map back GCE d} for each $x\in X$ and $f\in C(Y)$, we have $T(f)(x)\in\cconv{f(\Pi\inv(\{x\}))}$;
  \item\label{properties of map back GCE e} we have $T\adj(A\annih) \subseteq B\annih$;
  \item\label{properties of map back GCE f} we have $\Pi\dadj\circ T\adj = \id_{C(X)\dual}$ and $\Pi\dadj(B\annih) = A\annih$, where $\Pi\dadj = (\Pi\adj)\adj$.    
\end{enumerate}
In particular, if $A$ is non-trivial then $B$ is non-trivial.
\end{lemma}

An alternative proof of the final assertion is an easy modification of a similar proof for Cole extensions found in \cite{dawson2003t}.

Note that, if $B$ is a uniform algebra on a compact space $Y$ and $B$ is a generalised Cole extension of $A$, with associated maps $\Pi$ and $T$, then we have $B\cap \Pi\adj(C(X)) = \Pi\adj(A)$. This follows from clause \ref{properties of map back GCE e} above.

Cole extensions are generalised Cole extensions, by Proposition~\ref{simple Cole extension props}, but the converse is not true in general. We now give an example of a generalised Cole extension that is not a Cole extension; further examples will be given in the final section.

\begin{example}
  Let $A(\Dbar)$ be the disk algebra (here $\Dbar$ is the closed unit disk in $\C$). Then $A(\Dbar)\check\otimes A(\Dbar)$, the injective tensor product of $A(\Dbar)$ with itself, can be identified with a uniform algebra on $\Dbar\times \T$. Let $\rho:\Dbar\times\T\to\Dbar$ denote the coordinate projection onto the first coordinate, and let $\tau:\Dbar\to\T$ denote the function $\tau(z) = (z,1)$ $(z\in\Dbar)$. Then both $\rho$ and $\tau$ are continuous, and $\rho\circ\tau = \id_{\Dbar}$. Now we see that $A(\Dbar)\check\otimes A(\Dbar)$ is a generalised Cole extension of $A(\Dbar)$, with associated maps $\rho\adj:C(\Dbar)\to C(\Dbar\times\T)$ and $\tau\adj:C(\Dbar\times\T)\to C(\Dbar)$.

  To see that this does not give a Cole extension, take any $z\in\Dbar$. Then $\rho\inv(\{z\}) = \T$, which is connected. However, it is not hard to see that if $A(\Dbar)\check\otimes A(\Dbar)$ was a Cole extension of $A(\Dbar)$, then $\rho\inv(\{z\})$ would be (totally) disconnected. Also note that $A(\Dbar)\check\otimes A(\Dbar)$ is not natural on $\Dbar\times\T$, and that $\tau\adj$ is actually a homomorphism.
\end{example}

As for general uniform algebra extensions, a generalised Cole extension $B$ of $A$ need not be natural on the given compact space $Y$ even if $A$ is natural on $X$, as demonstrated by the above example. However, just as in the general case, we instead consider the Gelfand transform of $B$ when $A$ is natural on $X$. This is made precise in the following proposition; the proof is an easy exercise.

\begin{proposition}
Let $B$ be a uniform algebra on a compact space $Y$. Suppose that $A$ natural on $X$ and that $B$ is a generalised Cole extension of $A$ with associated maps $\Pi:Y\to X$ and $T:C(Y)\to C(X)$. Then $\widehat B$ is a generalised Cole extension of $A$ with associated maps $\widehat \Pi:\charspace B\to X$ and $\widehat T:C(\charspace B)\to C(X)$ given by $\widehat\Pi(\varphi) = \varphi\circ\Pi\adj$ $(\varphi\in\charspace B)$ and $\widehat T(f) = T(f|Y)$ $(f\in C(\charspace B))$.
\end{proposition}

Clearly the linear map associated with a generalised Cole extension need not be a homomorphism in general; one can easily show that the averaging map associated with a Cole extension is not multiplicative. The following proposition gives necessary and sufficient conditions for this to be the case.

\begin{proposition}\label{conditions for homomorphism}
Let $B$ be a uniform algebra on a compact space $Y$ and let $E\subseteq X$. Suppose that $B$ is a generalised Cole extension of $A$ with associated maps $\Pi$ and $T$. If, for each $x\in E$, $\mu_x$ is a representing measure for some character $\psi_x$ on $B$, then we have $T(fg)(x) = T(f)(x) T(g)(x)$ $(f,g\in B)$ for all $x\in E$. In particular, $T|B$ is a homomorphism if and only if, for each $x\in X$, $\mu_x$ is a representing measure for some character $\psi_x$ on $B$.
\end{proposition}
\begin{proof}
To prove the first assertion, suppose that, for each $x\in E$, $\mu_x$ is a representing measure for some character $\psi_x$ on $B$. Let $x\in E$ and let $f,g\in B$. Then we have
\[
T(fg)(x) = \int_Y fg\dm\mu_x = \psi_x(fg) = \psi_x(f)\psi_x(g) = T(f)(x)T(g)(x),
\]
which proves the first assertion. If the above holds with $E=X$, then it follows that $T(fg) = T(f)T(g)$ for all $f,g\in B$ and so $T|B$ is a homomorphism.

Now suppose that $T|B$ is a homomorphism. Let $\theta:=T|B$. Then $\theta\adj$ maps $\charspace A$ into $\charspace B$ and so, for each $x\in X$, there exists $\psi_x\in\charspace B$ such that $\theta\adj(\pe x) = \psi_x$. Then, for all $f\in B$, we have
\[
\psi_x(f) = \pe x(\theta(f)) = T(f)(x) = \int_Y f \dm\mu_x
\]
so that $\mu_x$ is a representing measure for $\psi_x$. This completes the proof.
\end{proof}

In the above proof, we necessarily have that the $\psi_x$ are distinct since $T|B$ is surjective; moreover, we have $\psi_x\circ\Pi\adj = \pe x$. The above result can be seen as a generalisation of \cite[Theorem~6.7]{wright1961}. Indeed, take $B=C(Y)$ and $A = C(X)$ in the above. Then the conclusion is $T$ is a homomorphism if and only if $\mu_x$ is a representing measure for $C(Y)$. But all representing measures for $C(Y)$ point mass measures, it follows that $\mu_x = \delta_y$ for some $y\in Y$. Thus $T$ is a homomorphism if and only if $T = \rho\adj$ for some continuous map $\rho:X\to Y$ with $\Pi\circ\rho = \id_{X}$.

In the setting of generalised Cole extensions, some results from Section~\ref{sec:uniform algebra extensions} can be strengthened.

\begin{proposition}
  Let $B$ be a uniform algebra on a compact space $Y$. Suppose that $B$ is a generalised Cole extension of $A$ with associated maps $\Pi:Y\to X$ and $T:C(Y)\to C(X)$. Let $E\subseteq X$ and $F = \Pi\inv(E)$. Then
\begin{enumerate}    
  \item\label{GCE property correspondence a} $E$ is a peak set for $A$ if and only if $F$ is a peak set for $B$;
  \item\label{GCE property correspondence b} $E$ is a peak set in the weak sense for $A$ if and only if $F$ is a peak set in the weak sense for $B$.    
\end{enumerate}    
\end{proposition}
\begin{proof}
  \ref{GCE property correspondence a}. By Proposition~\ref{peak set permanence}.\ref{peak set permanence a}, if $E$ is a peak set for $A$ then $F$ is a peak set for $B$. Conversely, if $F$ is a peak set for $B$ then there  exists $g\in B$ which peaks on $F$. Set $f:=T(g)$. Then, for each $x\in E$, we have
    \[
    f(x) = T(g)(x) \in \cconv{g(\Pi\inv(\{x\}))}\subseteq \cconv{g(F)}\subseteq\{1\}.
    \]
    For each $x\in X\setminus E$, we have $\abs{g(y)}<1$ for all $y\in \Pi\inv(\{x\})$, and since we have ${f(x)\in\cconv{g(\Pi\inv(\{x\}))}}$ it follows that $\abs{f(x)}<1$. Thus $f$ peaks on $E$ and so $E$ is a peak set for $A$. This completes the proof.

  \ref{GCE property correspondence b}. By Proposition~\ref{peak set permanence}.\ref{peak set permanence b}, if $E$ is a peak set in the weak sense for $A$, then $F$ is a peak set in the weak sense for $B$.

Suppose that $F$ is a peak set in the weak sense for $B$, and let $\mu$ be an annihilating measure for $B$. We must prove that $\mu_E$ is an annihilating measure for $A$. By Proposition~\ref{properties of map back GCE}.\ref{properties of map back GCE e}, $\nu:=T\adj(\mu)$ is an annihilating measure for $B$ such that $\pushforward\Pi\nu=\mu$. By the above statement, $\nu_F$ is also an annihilating measure for $B$. It is not hard to see that $\pushforward\Pi{\nu_F}=\mu_E$. Thus $\mu_E$ is an annihilating measure for $A$ and so $E$ is a peak set in the weak sense.
\end{proof}

The above indicates that generalised Cole extensions have a ``local nature''. This can be made precise in the following proposition.

\begin{proposition}
  Let $B$ be a uniform algebra on a compact space $Y$. Suppose that $B$ is a generalised Cole extension of $A$ with associated maps $\Pi:Y\to X$ and $T:C(Y)\to C(X)$. Let $K$ be a compact subset of $X$ and let $L = \Pi\inv(K)$. Let $\Pi_K = \Pi|L$ and let $T_K:C(L)\to C(K)$ be given by $T_K(f) = T(f)|K$ $(f\in B)$. Then $B_L$ is a generalised Cole extension of $A_K$ with associated maps $\Pi_K$ and $T_K$.
\end{proposition}
\begin{proof}
Clearly $\Pi_K$ is a continuous surjection from $L$ onto $K$, and $T_K$ is a unital linear map with $\norm{T_K} = 1$, and $T_K\circ\Pi_K\adj = \id_{C(K)}$. It remains to see that $T_K(B_L) = A_K$. Since $T_K(f|L) = T(f)|K$, for all $f\in B$, we have $T_K(f|L)\in A_K$ for all $f\in B$. Since $B|L$ is dense in $B_L$, it follows that $T_K(B_L)\subseteq A_K$. The reverse inclusion follows from the fact that $\Pi_K\adj(g|K) = \Pi\adj(g)|L$ for all $g\in A$ and that $A$ is dense in $A_K$. This completes the proof.
\end{proof}

Next we ask whether we can improve Theorem~\ref{shilov boundary extensions} for generalised Cole extensions. Using the linear map associated with a generalised Cole extension, we obtain the following characterisation of Shilov boundary points for an extension that are mapped into the Shilov boundary for $A$.

\begin{theorem}
Let $B$ be a uniform algebra on a compact space $Y$. Suppose that $B$ is a generalised Cole extension of $A$ with associated maps $\Pi:Y\to X$ and $T:C(Y)\to C(X)$. Let $y_0\in \shilov B$. Then $\Pi(y_0)\in\shilov A$ if and only if, for each open neighbourhood $U$ of $\Pi(y_0)$, there exist a peak set $K$ for $B$ and $x\in X$ such that $K\subseteq\Pi\inv(U)$ and $\mu_x(K)>0$.
\end{theorem}
\begin{proof}
First assume that, for each open neighbourhood $U$ of $\Pi(y_0)$, there exists a peak set $K$ for $B$ and $x\in X$ such that $K\subseteq \Pi\inv(U)$ and $\mu_x(K)>0$.

Let $x_0=\Pi(y)$ and fix an open neighbourhood $U$ of $x_0$. By assumption, there exists a peak set $K$ for $B$ and $x\in U$ such that $\mu_x(K)>0$. Let $f\in B$ be a function which peaks on $K$. Then the sequence $(f^n)$ is uniformly bounded and converges pointwise to $\chi_K$, the characteristic function for $K$. By the dominated convergence theorem, we have
\[
\lim_{n\to\infty}\int_Y f^n\dm\mu_x = \mu_x(K)>0
\]
Then there exists $n\in\N$, such that $\abs{f^n}_{Y\setminus \Pi\inv(U)}\leq \mu_x(K)/4$ and
\[
\abs[\Bigg]{\int_Y f^n\dm\mu_x} > \frac{3\mu_x(K)}{4}.
\]
Set $h:=T(f^n)$. Then we have
\[
\abs{h(x)} >  \frac{3\mu_x(K)}{4}
\]
and, for each $x'\in X\setminus U$, we have
\[
\abs{h(x)} \leq \mu_x(\Pi\inv(\{x\})\setminus K)\frac{\mu_x(K)}{4} \leq \frac{\mu_x(K)}{4}.
\]
Thus $\abs{h}_X>\abs{h}_{X\setminus U}$. It follows that $x\in \shilov A$.

Now assume that $x_0:=\Pi(y_0)\in\shilov A$. Let $U$ be an open neighbourhood $x_0$. Then there exists a peak set $L$ for $A$ such that $L\subseteq U$. Set $K:=\Pi\inv(L)$ and let $x\in L$. Then, by Lemma~\ref{peak set permanence}.\ref{peak set permanence a}, $K$ is a peak set for $B$ and ${K\subseteq\Pi\inv(U)}$. Moreover, ${\mu_x(K) = \mu_x(\Pi\inv(\{x\})) = 1}$. This completes the proof.
\end{proof}

In particular, if $B$, $Y$, $\Pi$, and $T$ are as in the above theorem and $\Pi\inv(\{x\})$ is finite for each $x\in X$, then $\shilov B = \Pi\inv(\shilov A)$. (This applies when $B$ is a Cole extension of $A$.)

\smallskip

Unfortunately, generalised Cole extensions have some limitations such as the following.

\begin{proposition}
Let $B$ be a uniform algebra on a compact space $Y$. Suppose that $B$ is a generalised Cole extension of $A$ with associated maps $\Pi:Y\to X$ and $T:C(Y)\to C(X)$. If $B$ is normal on $Y$ then $A$ is normal on $X$.
\end{proposition}
\begin{proof}
Let $K_1,K_2$ be closed subsets of $X$ with $K_1\cap K_2=\emptyset$. For $j=1,2$, set ${L_j:=\Pi\inv(K_j)}$. Then $L_1$ and $L_2$ are closed subsets of $Y$ with ${L_1\cap L_2=\emptyset}$. Since $B$ is normal on $Y$, there exists $g\in B$ such that $g(L_1)\subseteq\{1\}$ and $g(L_2)\subseteq\{0\}$. Set $f:=T(g)$. Then, for each $x\in K_1$, we have
    \[
    f(x) = T(g)(x) \in \cconv{g(\Pi\inv(\{x\}))}\subseteq \cconv{g(L_1)}\subseteq\{1\},
    \]
    and, for each $x\in K_2$, we have
    \[
    f(x) = T(g)(x) \in \cconv{g(\Pi\inv(\{x\}))}\subseteq \cconv{g(L_2)}\subseteq\{0\}.
    \]
    Thus $f(K_1)\subseteq\{1\}$ and $f(K_2)\subseteq\{0\}$. It follows that $A$ is normal on $X$.
\end{proof}

In particular, we cannot hope to construct new examples of normal uniform algebras using generalised Cole extensions. See the comments following \cite[Proposition~7]{feinstein2001}, which is similar to the above; the same proof is used to prove the result for Cole extensions. A similar proof can be used to show that if $B$ has spectral synthesis then $A$ has spectral synthesis.

\subsection{Groups associated with generalised Cole extensions}
The subclass of generalised Cole extensions that are ``induced'' by the action of a compact group are especially interesting. Indeed, many of the results from the previous sections can be strengthened further in this setting, and many examples of generalised Cole extensions that exist in the literature belong to this subclass. First we recall some standard definitions and results on compact groups and continuous group actions. We refer the reader to \cite{folland2016, hewitt1979} for more information on these topics.

A \emph{compact group} is a group $G$ equipped with a compact, Hausdorff topology such that the maps $(s,t)\mapsto st:G\times G\to G$ and $s\mapsto s\inv:G\to G$ are continuous. Let $G$ be a compact group. Then there is a unique probability measure $\mu$ on $G$ such that $\mu(sE)=\mu(E)=\mu(Es)$ for all $E\subseteq G$ and $s\in G$. (Here $sE$ is defined to be $\{st:t\in E\}$, and $Es$ is defined in a similar manner.) This measure $\mu$ is called the \emph{normalised Haar measure} for $G$. We say that $G$ \emph{acts continuously} on a topological space $X$ if the action is jointly continuous.

Let $G$ be a compact group and $X$ a compact space, and suppose that $G$ acts continuously on $X$. For each $E\subseteq G$ and $F\subseteq X$, let $E\cdot F = \{s\cdot x:s\in E,x\in F\}$; in the special case where $E = \{s\}$ for some $s\in G$, we instead write $s\cdot F$; $E\cdot x$, $F\cdot E$, $F\cdot s$, and $x\cdot E$ are defined analogously. We say that a subset $E$ of $X$ is $G$-invariant if $G\cdot E = E$, and the orbit of $x\in X$ with respect to $G$ is the set $G\cdot x$, and the set of all $G$-orbits in $X$ is denoted $X/G$.

We begin with the following fundamental lemma, which is certainly well known. In several examples from the literature where generalised Cole extensions appear, some version of this lemma are proved for the specific case. We provide a proof for the general case for the convenience of the reader.

\begin{lemma}\label{averaging map from Haar integral}
Let $X,Y$ be compact spaces, let $\Pi:Y\to X$ be a continuous surjection, let $G$ be a compact group, and let $\mu$ be then normalised Haar measure on $G$. Suppose that $G$ acts continuously on $Y$ and that, for each $x\in X$ and $y\in Y$ with $\Pi(y)=x$, we have $G\cdot y=\Pi\inv(\{x\})$. For each $f\in C(Y)$, let
\begin{equation}\label{group integral averaging map}
P(f)(y) = \int_G f(s\cdot y)\dm\mu(s).
\end{equation}
Then $P$ is a continuous projection on $C(Y)$ such that $P\circ\Pi\adj=\Pi\adj$ and $\norm{P}=1$. Moreover, $P$ is averaging and the range of $P$ coincides with the set $\Pi\adj(C(X))$.
\end{lemma}
\begin{proof}
We must first show that $P(f)$ is a continuous function on $X$. This part of the proof is effectively \cite[Theorem~19.1]{stout1971}.
Fix $\varepsilon>0$. For each ${(s,y)\in G\times Y}$, there exist open neighbourhoods $U$ of $s$ and $V$ of $y$ such that for all ${(s',y')\in U_s\times V_y}$ we have $\abs{f(s'\cdot y')-f(s\cdot y)}<\varepsilon/2$. Fix $y_0\in G$. By compactness, there exist finitely many pairs ${(s_1,y_1),\dots,(s_n,y_n)\in G\times Y}$ such that ${G\times \{y_0\}\subseteq \bigcup_{j=1}^n (U_{s_j}\times V_{y_j})}$. Set $U_j := U_{s_j}$ and $V_j := V_{y_j}$ for $j=1,\dots,n$. Let $V:=\bigcap_{j=1}^n V_j$ so that $V$ is an open neighbourhood of $y_0\in Y$. For each $j\in\{1,\dots,n\}$, put $U_j':=U_j\setminus \bigcup_{k=1}^{j-1} U_{k}$. Then for each $y\in V$ we have
\[
\int_G \abs{f(s\cdot y) - f(s\cdot y_0)}\dm\mu(s)\leq \sum_{j=1}^n \int_{U_j'}\abs{f(s\cdot y) - f(s\cdot y_0)}\dm\mu(s)<\varepsilon.
\]
It follows that $\abs{P(f)(y) - P(f)(y_0)}<\varepsilon$ for all $y\in V$. This shows that $P(f)$ is continuous on $Y$.

To see that $P$ is a projection, take $f\in C(Y)$ and $y\in Y$. Then
\begin{align*}
\int_G P(f)(s\cdot y)\dm\mu(s) &= \int_G \int_G f(t\cdot (s\cdot y))\dm\mu(t) \dm\mu(s)\\
&= \int_G\int_G f(s\cdot y)\dm\mu(s)\dm\mu(t)\\
&= \int_G P(f)(y)\dm\mu(t)\\
&= P(f)(y)
\end{align*}
by the right invariance of the normalised Haar measure.

Next, for each $f\in C(Y)$ and $y\in Y$, we have
\[
\abs{P(f)(y)}\leq \int_G \abs{f(s\cdot y)}\dm\mu(s) \leq \int_G\abs{f}_Y\dm\mu(s) = \abs{f}_Y.
\]
Thus $\abs{P(f)}_Y\leq \abs{f}_Y$. It is clear that $P(1)=1$, so that $\norm{P}=1$.

It is trivial to see that $\norm{P}=1$. To see that $P(C(Y)) = \Pi\adj(C(X))$, consider $f\in C(Y)$. Then, for each $y\in Y$ and $s\in G$, we have
\[
P(f)(s\cdot y) = \int_G f(t\cdot (s\cdot y))\dm\mu(t) = \int_G f(t\cdot y)\dm\mu(t) = P(f)(y)
\]
by the right invariance of $\mu$. In particular $P(f)$ is constant on the sets $\Pi\inv(\{x\})$ ${(x\in X)}$, and so there is a unique continuous function ${g\in C(X)}$ such that ${f = g \circ \Pi}$. Thus we have ${P(f)\in\Pi\adj(C(X))}$. The fact that $P$ is averaging now follows from Lemma~\ref{averaging equivalences} (applied to the map $(\Pi\adj)\inv\circ P$). This completes the proof.
\end{proof}

Throughout the remainder of this section, let $X, Y$ be compact spaces, let $A$ be a uniform algebra on $X$, let $B$ be a uniform algebra on $Y$, let $G$ be a compact group and let $\mu$ denote the normalised Haar measure on $G$.

\begin{definition}
Suppose that $B$ is a generalised Cole extension of $A$ with associated maps $\Pi:Y\to X$ and $T:C(Y)\to C(X)$. We say that this extension is \emph{implemented by $G$} if $G$ acts continuously on $Y$; if ${G\cdot \Pi(y)=\Pi\inv(\{\Pi(y)\})}$ for all $y\in Y$; and if, for each $f\in B$, we have
\begin{equation}\label{implemented by group averaged function}
T(f)(\Pi(y)) = \int_G f(s\cdot y)\dm\mu(s)\qquad(y\in Y).
\end{equation}
\end{definition}

Cole's counterexample to the peak point conjecture is a generalised Cole extension which is implemented by a group.

The following lemma provides some useful facts about generalised Cole extensions that are implemented by groups.

\begin{lemma}\label{useful group imp}
Suppose that $B$ is a generalised Cole extension of $A$ with associated maps $\Pi:Y\to X$ and $T:C(Y)\to C(X)$, and that $G$ acts continuously on $Y$.
\begin{enumerate}
	\item\label{useful group imp a} If $G\cdot y = \Pi\inv(\{\Pi(y)\})$ for each $y\in Y$, then $\Pi$ is an open map.
	\item\label{useful group imp b} If
	\begin{equation}\label{useful group imp b eqn}
  	T(f)(\Pi(y)) = \int_G f(s \cdot y) \dm\mu (s)\qquad (y\in Y)
  	\end{equation}
  	for all $f\in C(Y)$, then $G\cdot y = \Pi\inv(\{\Pi(y)\})$ for all $y\in Y$.
\end{enumerate}
\end{lemma}
\begin{proof}
\ref{useful group imp a}. Suppose that we have $G\cdot y = \Pi\inv(\{\Pi(y)\})$ for all $y\in Y$. Let $q:Y\to Y/G$ denote the quotient map. Then $q$ is continuous and open, and $Y/G$ is homeomorphic to $X$, via the map $\tilde\Pi:Y/G\to X:\Pi\inv(\{x\}) \mapsto x$, and $\Pi = \tilde\Pi\circ q$. Thus $\Pi$ is an open map.

\ref{useful group imp b}. First suppose that \eqref{useful group imp b eqn} holds for all $f\in C(Y)$. Fix $x\in X$, let $y\in \Pi\inv(\{x\})$, and let $F$ be the orbit of $y$ under $G$. Suppose, towards a contradiction, that $F\not\subseteq \Pi\inv(\{x\})$. Then there exists $s_0\in G$ such that $s_0\cdot y\notin \Pi\inv(\{x\})$.

    By Urysohn's lemma, there is a function $f\in C(Y)$ with $f(Y)\subseteq[0,\infty)$ such that $f(\Pi\inv(\{x\}))\subseteq \{0\}$ and $f(s_0\cdot y) = 1$. Since the action of $G$ on $Y$ is jointly continuous, there is an open neighbourhood $U$ of $s_0$ in $G$ such that $\abs{f(s\cdot y) - 1} < 1/4$ whenever $s\in U$. 
    Then we have
    \[
    T(f)(\Pi(y)) = \int_G f(s\cdot y) \dm\mu(s) \geq \int_U f(s\cdot y) \dm\mu(s) > \frac{\mu(U)}{4} > 0. 
    \]
    But since $f(\Pi\inv(\{x\}))\subseteq \{0\}$ and $T(f)(x)\in\cconv{f(\Pi\inv(\{x\}))}\subseteq\{0\}$, we have a contradiction. Hence $F\subseteq \Pi\inv(\{x\})$.

    Now suppose, again towards a contradiction, that $F\neq \Pi\inv(\{x\})$. Take $y_0\in\Pi\inv(\{x\})\setminus F$ and let $K$ be the orbit of $y_0$ in $Y$. (By the above, $K\subseteq \Pi\inv(\{x\})$.) Then there exists $f\in C(Y)$ such that $f(F)\subseteq\{0\}$ and $f(K)\subseteq \{1\}$. Hence
    \[
    1 = \int_G f(s\cdot y_0)\dm \mu(s) = T(f)(\Pi(y_0)) = T(f)(\Pi(y)) = \int_G f(s\cdot y)\dm\mu(s) = 0 
    \]
    which is a contradiction. Hence $F = \Pi\inv(\{x\})$ and the result follows. This completes the proof.    
\end{proof}

Suppose that $G$ acts continuously on $Y$, then for each $f\in C(Y)$ and $s\in G$ we can form a new function $f_s\in C(Y)$ given by $f_s(y) = f(s\cdot y)$ $(y\in Y)$. 

\begin{proposition}\label{implementation conditions}
Suppose that $G$ acts continuously on $Y$ and that $\Pi:Y\to X$ is a continuous surjection such that $G\cdot y = \Pi\inv(\{\Pi(y)\})$ for all $y\in Y$. If $f_s\in B$ for all $f\in B$ and $s\in G$ and $\Pi\adj(C(X))\cap B = \Pi\adj(A)$, then $B$ is a generalised Cole extension of $A$ implemented by $G$.
\end{proposition}
\begin{proof}
Suppose that $\Pi\adj(C(X))\cap B = \Pi\adj(A)$ and that, for all $f\in B$ and $s\in G$, we have $f_s\in B$. Let $T:C(Y) \to C(X)$ be the norm $1$ linear map given by
\[
T(f)(\Pi(y)) = \int_G f(s\cdot y)\dm\mu(s)\qquad (y\in Y, f\in C(Y)).
\]
Then $T$ is a unital and, since $G\cdot y = \Pi\inv(\{\Pi(y)\})$ for all $y\in Y$, it follows that $T\circ\Pi\adj = \id_{C(X)}$. It remains only to show that $T(B) = A$. Let $f\in B$ and let $g = T(f)\in C(X)$. Then, if $\nu$ is an annihilating measure for $B$, then by Fubini's theorem
\[
\int g(\Pi(y))\dm\nu(y) = \int_Y\int_G f_s(y)\dm\mu(s)\dm\nu(y) = \int_G\int_Y f_s(y)\dm\nu(y)\dm\mu(s) = 0.
\]
It follows that $\Pi\adj(g)\in B\cap \Pi\adj(C(X))$ and so $g\in A$. This completes the proof.
\end{proof}

Note that the condition $\Pi\adj(C(X))\cap B = \Pi\adj(A)$ in the above is necessary, as mentioned below Proposition~\ref{properties of map back GCE}. The condition that $f_s\in B$ whenever $f\in B$ and $s\in G$, in the above, is seemingly stronger than \eqref{implemented by group averaged function}. For generalised Cole extensions that satisfy this additional condition we can apply results of Bj\"ork \cite{MR0328604} and Izzo \cite{MR2648073} to obtain the following proposition.

\begin{proposition}\label{gceg extra properties}
Suppose that $B$ is a generalised Cole extension implemented by $G$, with associated maps $\Pi:Y\to X$ and $T:C(Y)\to C(X)$. Suppose that, for all $f\in B$ and $s\in G$, we have $f_s\in B$.
\begin{enumerate}
	\item\label{gceg extra properties a} The action of $G$ on $Y$ can be extended to a continuous action of $G$ on $\charspace B$.
	\item\label{gceg extra properties b} For each $x\in X$, if $B_{\Pi\inv(\{x\})}$ is natural on $\Pi\inv(\{x\})$ then $B_{\Pi\inv(\{x\})} = C(\Pi\inv(\{x\}))$.
	\item\label{gceg extra properties c} If $A$ is natural on $X$, $\Pi\adj(A)$ separates the $G$-orbits in $\charspace{B}$, and, for each $x\in X$, $B_{\Pi\inv(\{x\})}$ is natural on $\Pi\inv(\{x\})$, then $B$ is natural on $Y$.
	\item\label{gceg extra properties d} We have $B = C(Y)$ if and only if $B$ is natural on $Y$ and $A = C(X)$.
	\item\label{gceg extra properties e} We have $\shilov B = \Pi\inv(\shilov A)$.
\end{enumerate}
\end{proposition}
\begin{proof}
\ref{gceg extra properties b}. For each $s\in G$ and $\varphi\in\charspace B$, let $\varphi_s$ be the character on $B$ given by $\varphi_s(f) = \varphi(f_s)$ $(f\in B)$. This action is obviously continuous and agrees with the action of $G$ when restricted to $Y$.

\ref{gceg extra properties b}. Let $x\in X$. Suppose that $B_{\Pi\inv(\{x\})}$ is natural on $\Pi\inv(\{x\})$. Since $\Pi\inv(\{x\})$ is a $G$-orbit, it follows that $f_s\in B_{\Pi\inv(\{x\})}$ whenever $f\in B_x$. The result follows from \cite[Corollary~1.4]{MR2648073}.

\ref{gceg extra properties c}. Suppose that $A$ is natural on $X$ and that, for each $x\in X$, $B_{\Pi\inv(\{x\})}$ is natural on $\Pi\inv(\{x\})$. Then, by \cite[Theorem~1.2]{MR0328604}, there is a homeomorphism $\tau:X\to \charspace B/G$, where the action of $G$ on $\charspace B$ is as in the proof of clause \ref{gceg extra properties a}. Let $\varphi\in \charspace B$. Then there is a unique $x\in X$ such that $\tau(G\cdot \varphi) = x$. It follows that $\varphi$ belongs to the character space of $B_{\Pi\inv(\{x\})}$, which is $\Pi\inv(\{x\})$, and so $\varphi$ is a point evaluation at some $y\in\Pi\inv(\{x\})$. The result follows.

\ref{gceg extra properties d}. If $B = C(Y)$ then $B$ is natural on $Y$ and $A = C(X)$ (by Proposition~\ref{properties of map back GCE}). Conversely, suppose that $B$ is natural on $Y$ and $A =C(X)$. Then each point $x\in X$ is a strong boundary point for $A$ and so $\Pi\inv(\{x\})$ is a peak set in the weak sense, by Proposition~\ref{GCE property correspondence b}. Hence $B_{\Pi\inv(\{x\})}$ is closed in $C(\Pi\inv(\{x\}))$ for all $x\in X$. Now if $B\neq C(Y)$, then there exists $x\in X$ such that $B_{\Pi\inv(\{x\})}\neq C(\Pi\inv(\{x\}))$ and, by \ref{gceg extra properties b}, this implies that $B_{\Pi\inv(\{x\})}$ is not natural on $\Pi\inv(\{x\})$ for some $x\in X$. But since $B$ is natural on $Y$, it follows that $B_{\Pi\inv(\{x\})}$ is natural on $\Pi\inv(\{x\})$. This is a contradiction. Hence $B = C(Y)$.

\ref{gceg extra properties e}. This follows from \cite[Theorem~1.2]{MR0328604} using a similar argument to that in the proof of \ref{gceg extra properties c}.
\end{proof}

In general, a generalised Cole extension of $A$ need not be implemented by a group. However, in a special case it is possible to show that there is a finite group that implements the extension. In fact, a generalised Cole extension of $A$ that satisfies this extra condition and which is generated by ``adjoining'' a single function to $A$ is necessarily a Cole extension of $A$.

\begin{theorem}\label{thm when gce is ce}
Suppose that $B$ is a generalised Cole extension of $A$ with associated maps $\Pi:Y\to X$ and $T:C(Y)\to C(X)$. Suppose further that ${\norm{\id_{C(Y)} - \Pi\adj\circ T} = 1}$. Then there exists a finite group $G$ such that $B$ is a generalised Cole extension of $A$ implemented by $G$. Moreover, if there exists $h_0\in B$ with $T(h_0) = 0$ such that $B$ is the uniform algebra generated by $\Pi\adj(A)$ and $h_0$ then there exists a homeomorphism ${\psi:Y\to X^q}$, where $q(t) = t^2 - h$ for some $h\in A$, such that $\Pi_q\circ\psi = \Pi$, $\psi\adj(A^q) = B$, and $\psi\adj(p_q) = h_0$. 
\end{theorem}
\begin{proof}
Let $P := \Pi\adj\circ T$ and $Q := \id_{C(Y)} - P$. Then $P$ is a unital projection on $C(Y)$ and $\norm{P} = \norm{Q} = 1$. Set $\theta := P - Q$, so that $P = (\id_{C(Y)} + \theta)/2$. By \cite[Lemma~4.1]{blecher2015}, $\theta$ is a homomorphism with $\theta\circ\theta = \id_{C(Y)}$. Hence $\theta$ is an automorphism and is isometric by \cite[Proposition~1.5.28]{dales2000} (this fact is contained in \cite[Theorem~4.5]{blecher2015}). Thus there is a homeomorphism $\rho:Y\to Y$ such that $\theta = \rho\adj$, and we have $P = (\id_{C(Y)} + \theta)/2$.

Let $G := \{\id_Y, \rho\}$; under composition and given the discrete topology this is a compact group. Let $\mu$ denote normalised Haar measure on $G$. Note that $\mu(\{\id_Y\}) = \mu(\{\rho\}) = 1/2$, and $G$ acts continuously on $Y$. Thus, for each $f\in C(Y)$, we have
\[
\int_G f(s\cdot y) \dm \mu(s) = \frac12 (f(y) + f(\rho(y))) = P(f)(y)\qquad (y\in Y).
\]
By Lemma~\ref{useful group imp}.\ref{useful group imp b}, we have $G\cdot y = \Pi\inv(\{\Pi(y)\})$ for all $y\in Y$ and it follows that $B$ is a generalised Cole extension of $A$ implemented by $G$.

Suppose now that there exists $h_0\in B$ with $T(h_0) = 0$ and such that $B$ is generated by $\Pi\adj(A)$ and $h_0$. Again by \cite[Lemma~4.1]{blecher2015}, we have $h_0^2\in \Pi\adj(A)$, so there exists $h\in A$ such that $h_0^2 = \Pi\adj(h)$. Set $q(t) = t^2 - h$. We must now prove that $B$ is isomorphic to $A^q$.

Define the map $\psi:Y\to X\times X^q$ by $\psi(y) = (\Pi(y),h_0(y))$. It is not hard to see that $\psi$ is surjective, and it is immediate that $\Pi_q\circ\psi = \Pi$. We \emph{claim} that $\psi$ is injective. Indeed, suppose that $y_1,y_2\in Y$ with $y_1\neq y_2$ and $\psi(y_1) = \psi(y_2)$. Then we have $\Pi(y_1)=\Pi(y_2)$ and $h_0(y_1)=h_0(y_2)$. We have $y_1=\rho(y_2)$ and so $h_0(y_1) = h_0(y_2) = -h_0(y_1)$. Hence $h_0(y_2) = h_0(y_1) = 0$. This contradicts the fact that $B$ separates the points of $Y$, and so no such points can exist. Hence $\psi$ is injective, and so is a homeomorphism.

The function $\psi\adj:C(X^q)\to C(Y)$ is an isometric isomorphism such that $\Pi\adj = \psi\adj\circ\Pi_q\adj$. Finally, we must prove that $\psi\adj(A^q) = B$. Since $p_q\circ\psi(y) = h_0(y)$ for each $y\in Y$, so that $\psi\adj(p_q) = h_0$, and $B$ is generated by $\Pi\adj(A)$ and $h_0$ and $A^q$ is generated by $\Pi_q\adj(A)$ and $p_q$. This completes the proof.
\end{proof}

Given two functions $g_1,g_2\in A$ with disjoint support, and then form Cole extensions $(A^{q_1})^{q_2}$, where $q_1(t) = t^2 - g_1$ and $q_2(t) = t^2 - \Pi_{q_1}\adj(g_2)$, then it is not hard to see that this uniform is not generated by the embedding of $A$ and a single element, but does satisfy all the other conditions of the above. We have not been able to show that all generalised Cole extensions that satisfy the conditions above are necessarily the Cole extensions, although we conjecture that this is the case.

In the above, we do not use the full power of \cite[Lemma~4.1]{blecher2015}, and so our result can be slightly generalised to obtain the following theorem.

\begin{theorem}
Suppose that $P$ is a unital projection on $A$ such that both $P$ and $\id_A - P$ have norm $1$. Let $A_0 = P(A)$. Then there is a finite group $G$ such that $G$ acts continuously on $\charspace{A}$. If $A_0$ separates the $G$-orbits in $\charspace A$, then $A_0$ can be viewed as a natural uniform algebra on $\charspace A/G$. Moreover, in this case, $A$ is a generalised Cole extension of $A_0$ implemented by $G$. If, in addition, $X$ is a $G$-invariant subset of $\charspace A$, then $A_0$ is a uniform algebra on $X/G$.
\end{theorem}
\begin{proof}
As in the proof of Theorem~\ref{thm when gce is ce}, let $\theta = P - Q$, where $Q = \id_A - P$. Then $\theta$ is an (isometric) automorphism of $A$ with $\theta\circ\theta = \id_{A}$, and hence there is a homeomorphism $\rho$ of $\charspace A$ with $\rho\adj|A = \theta$. Let $G = \{\id_{\charspace A},\rho\}$. It is easy to see that $G$ acts continuous on $\charspace A$.

Suppose now that $A_0$ separates the $G$-orbits in $\charspace A$. Let $q:\charspace A\to \charspace A/G$ be the quotient map. Then we have $G\cdot \varphi = q\inv(q(\varphi))$ for all $\varphi\in \charspace A$, and $A_0$ consists precisely of those functions $f$ in $A$ such that $\widehat f$ is constant on each $G$ orbit in $\charspace A$, and so $A_0$ can be viewed as a natural uniform algebra on $\charspace A/G$, and $q\adj(C(\charspace A/G))\cap A = A_0$. By Proposition~\ref{implementation conditions}, $A$ is a a generalised Cole extension of $A_0$ implemented by $G$.
\end{proof}

\subsection{Examples of generalised Cole extensions}
\label{sec:examples of gces}
We now give some examples of generalised Cole extensions from the literature. First, we note that the bidual of a generalised Cole extension of a uniform algebra $A$ is a generalised Cole extension of $A\ddual$.

Let $B$ be a uniform algebra on a compact space $Y$. Suppose that $B$ is a generalised Cole extension of $A$ with associated maps $\Pi$ and $T$. It is interesting to note that $C(Y)$ is also a generalised Cole extension of $C(X)$, with associated maps $\Pi$ and $T$, and that $B\ddual$ (the bidual of $B$) is a generalised Cole extension of $A\ddual$. In the latter, the associated maps are given by the map ${\tau:\charspace{C(Y)\ddual}\to\charspace{C(X)\ddual}}$ induced by ${\pi^{***}:C(X)\ddual\to C(Y)\ddual}$, and the linear map ${T\dadj:C(Y)\ddual\to C(X)\ddual}$. (Note that bidual of a uniform algebra is again a uniform algebra.) We refer the reader to \cite{cole1982} for similar comments and results.

\subsubsection{Basener's counterexample to the peak point conjecture}
In \cite{basener1971}, Basener gave a simple counter example to the peak point conjecture using a compact subset of $\C^2$. We now show that this example can be realised as a generalised Cole example. We refer the reader to \cite[Example~19.8]{stout1971} for a detailed account of the construction.

If $X$ is a compact subset of $\C^n$, for some $n\in\N$, then $R(X)$ denotes the closure in $C(X)$ of the algebra of all restrictions to $X$ of rational functions on with no poles on $X$. It is clear that $R(X)$ is a uniform algebra on $X$ for each such $X$. If $X$ is a compact subset of $\C$, then $R(X)$ is always natural on $X$.

For the remainder of the section, fix a compact subset $K$ in the open unit disk in $\C$ such that $R(K)$ is non-trivial and $K$ does not admit any non-trivial Jensen measures. Set $A:=R(K)$ and set
\[
X_K:=\{(z,w)\in\C^2: z\in K, \abs{z}^2+\abs{w}^2 = 1\}.
\]
Then $X_K$ is a non-empty, compact subset of $\C^2$. (In fact, it is a compact subset of the unit sphere in $\C^2$.) Set $B:=R(X_K)$. Let $\Pi:X_K\to K$ and $p:X_K\to\C$ be the restrictions of the coordinate projections. 
Clearly $B$ is a uniform algebra extension of $A$. For each $z\in K$, we have
\[
\Pi\inv(\{z\})=\Bigl\{w\in\C: \abs{w} = \sqrt{\cramped{1-\abs{z}^2}}\Bigr\},
\]
which is a circle in $\C$. In particular $B|\Pi\inv(\{z\})$ is dense in $C(\Pi\inv(\{z\}))$ for each $z\in K$ since $B$ contains $p$ and $p\inv$. In fact, $B$ is generated by the functions $p$, $p\inv$, and $\Pi\adj(g)$ $(g\in A)$. It is a trivial fact that every point of $X_K$ is a peak point for $B$. In particular, $\shilov B = X_K$. Some additional features of the construction are outlined in the following proposition. We give a proof of these facts using the techniques of this paper; a direct proof of clause \ref{basener prop a} is given in \cite{stout1971}.

\begin{proposition}\label{basener prop}
\begin{enumerate}
	\item\label{basener prop a} The algebra $B$ is natural on $X_K$.
	\item\label{basener prop b} The circle group $\T$ acts continuously on $X_K$, via the formula 
	\[
	s\cdot (z,w) = (z,sw) \qquad (s\in \T, (z,w)\in X_K).
	\]
	\item\label{basener prop c} The algebra $B$ is a generalised Cole extension of $A$ implemented by $\T$.
\end{enumerate}
\end{proposition}
\begin{proof}
\ref{basener prop a}. Since $B|\Pi\inv(\{z\})$ is dense in $C(\Pi\inv(\{z\}))$ for all $z\in K$, the result follows from Theorem~\ref{Jensen measure triviality thm}.

\ref{basener prop b}. This is clear from the definition.

\ref{basener prop c}. The $\T$-orbits in $X_K$ are precisely the fibres of $\Pi$, by definition. Clearly, for all $s\in G$, we have $p_s = sp\in B$ and $(p\inv)_s = s\inv p\inv\in B$ and $(\Pi\adj(g))_s =\Pi\adj(g)$ for all $g\in A$ and $s\in G$. Since $B$ is generated by the functions $p$, $p\inv$, $\Pi\adj(g)$ ($g\in A$), it quickly follows that $f_s\in B$ for all $f\in B$ and $s\in G$. Fix $(z,w)\in X_K$ and $n\in\Z$. Then we have
\[
\int_\T [ws]^n \dm\mu(s) = \frac{1}{2\pi} \int_{-\pi}^{\pi}\biggl[\sqrt{1 - \abs{z}^2}\re^{\ri(\theta + t)}\biggr]^n\dm t,
\]
where $\mu$ denotes normalised Lebesgue measure on $\T$, and $\theta$ is the principle argument of $w$. It follows that
\[
\int_\T [ws]^n \dm\mu(s) = \frac{1}{2\pi\ri}\int_\gamma \zeta^{n-1}\dm\zeta
\]
where $\gamma$ is the path describing the circle with radius $\sqrt{1 - \abs{z}^2}$ once anticlockwise. Cauchy's theorem now tells us that this integral is zero whenever $n\neq 0$, and it is $1$ if $n=0$. It follows that $\Pi\adj(C(K))\cap B = \Pi\adj(A)$. Hence $B$ is a generalised Cole extension of $A$ implemented by $\T$ by Proposition~\ref{implementation conditions}.
\end{proof}

In fact, in the above, we showed that $B$ satisfies the stronger property that $f_s\in B$ for all $f\in B$ and $s\in \T$.

\subsubsection{Factorisation in commutative Banach algebras}
A more recent construction of a generalised Cole extension is given in \cite{dales2018}, where uniform algebras in which certain null sequences factor. We summarise some of the details of their construction below.

Let $A$ be a uniform algebra on a compact space $X$, fix $x_0\in X$, and let $M_{x_0}$ be the maximal ideal of all functions in $A$ that vanish at $x_0$. Fix a null sequence $(f_n)\in M_{x_0}$. Let $Y$ be the set of all triples $(x, (z_n), w)\in X\times\C^\N\times \C$ satisfying the following conditions:
\begin{enumerate}
	\item $z_nw = f_n(x)$ $(n\in\N)$;
	\item $\abs{w} = \max\{\abs{f_n(x)}^{1/2}:n\in\N\}$;
	\item $\abs{z_n}^2\leq \abs{f_n(x)}$ $(n\in\N)$.
\end{enumerate}
Then $Y$ is a compact space. Let $B$ be the uniform algebra on $Y$ generated by the functions $\Pi\adj(f)$, $p_n$ $(n\in\N)$, and $q$, where $\Pi$, $p_n$ $(n\in\N)$, and $q$ are the obvious coordinate projections. Clearly $B$ is a uniform algebra extension of $A$. In fact, $B$ is a generalised Cole extension of $A$ implemented by $\T$.  Moreover, the fibre of $\Pi$ over $x_0$ consists of a single point. (Extensions such as the above where there is a point $x_0$ whose fibre consists of a single point are called \emph{distinguished point extensions} in \cite{dales2018}.)

Note that, as for Basener's construction, this construction also satisfies the  stronger condition that $f_s\in B$ for all $f\in B$ and $s\in \T$.

\section*{Open problems}
We conclude with two open problems. In the following $X$ and $Y$ are compact spaces, $A$ is a uniform algebra on $X$, and $B$ is a uniform algebra on $Y$.

\begin{question}
Suppose that $B$ is a uniform algebra extension of $A$ with associated map $\Pi:Y\to X$, and that $P$ is a continuous projection on $B$ with $P(B)=\Pi\adj(A)$. Is $B$ a generalised Cole extension of $A?$
\end{question}

\begin{question}
Suppose that $B$ is a generalised Cole extension of $A$. When is this extension implemented by a compact group $G$ acting continuously on $Y?$
\end{question}

\begin{acknowledgements}
I would like to thank Dr.~J.~Feinstein for many helpful discussions on uniform algebras and Cole extensions, and for suggesting the topic of generalised Cole extensions to me.
\end{acknowledgements}

\bibliographystyle{abbrvnat}

\end{document}